\documentclass[11pt]{article}
\usepackage{amsmath,amsthm,verbatim,amssymb,amsfonts,amscd, graphicx}
\usepackage{hyperref}
\usepackage{parskip}

\usepackage[utf8]{inputenc}    
\usepackage[T1]{fontenc}

\theoremstyle{plain}
\newtheorem{theorem}{Theorem}[section]
\newtheorem{lemma}{Lemma}[section]
\newtheorem{corollary}{Corollary}[section]
\newtheorem{proposition}{Proposition}[section]

\newtheorem*{claim*}{Claim}
\newtheorem*{lemma*}{Lemma}
\newtheorem*{theorem*}{Theorem}

\theoremstyle{definition}
\newtheorem{definition}{Definition}[section]

\theoremstyle{remark}

\DeclareMathOperator{\Hessian}{Hess}

\DeclareMathOperator{\support}{supp}
\DeclareMathOperator{\divergence}{div}
\DeclareMathOperator{\Dom}{Dom}

\renewcommand{\Im}{\operatorname{Im}}
\renewcommand{\Re}{\operatorname{Re}}
\renewcommand{\i}{\operatorname{\sqrt{-1}}}

\allowdisplaybreaks

\usepackage[
backend=biber,  
natbib=true,
url=false, 
]{biblatex}
\IfFileExists{mybibliography.bib}
{\addbibresource{mybibliography.bib}}
{}
\IfFileExists{../bib/mybibliography.bib}
{\addbibresource{../bib/mybibliography.bib}}
{}

\begin{document}
	
	\title{The Diederich--Forn\ae ss index and the regularities on the $\bar{\partial}$-Neumann problem}
	
	\author{
		Bingyuan Liu\\
		bl016@uark.edu
	}
	
	\date{\today}

	\maketitle
	
	\begin{abstract}
		We show, under an assumption on the weakly pseudoconvex points, the trivial Diederich--Forn\ae ss index directly implies the global regularities of the $\bar{\partial}$-Neumann operators. 
	\end{abstract}
	
	\section{Introduction}
	The $\bar{\partial}$ equation $\bar{\partial}u=f$,  also known as ``the inhomogeneous Cauchy--Riemann equation'' is the most fundamental equation in the field of complex analysis. For the case of one complex variable, it has been well-understood for more than a century. Indeed, the explicit solution can be written by the Cauchy integral formula. But in terms of several complex variables, the satisfactory answer was not available until the 1960s. Around 1965, H\"{o}rmander \cite{Ho65} (see Andreotti--Vesentini \cite{AV61}, \cite{AV65} as well) showed that on bounded pseudoconvex domains, the $\bar{\partial}$ equation is always solvable for a $\bar{\partial}$-closed $f$. Moreover, the solution $u$ is smooth if $f$ is smooth (see Theorem 4.5.1 in Chen--Shaw \cite{CS01}). The result was obtained through a $L^2$ estimate: the Morrey--Kohn--H\"{o}rmander basic estimate. 
	
	As a byproduct, the Morrey--Kohn--H\"{o}rmander basic estimate also shows the Poisson's equation $\Box u=f$, with a natural but uncommon boundary condition, has a solution in the sense of distribution. Indeed, Garabedian--Spencer \cite{GS52} suggested studying this problem as a generalization of Hodge theory from compact manifolds to manifolds with boundaries. This problem is nowadays called the $\bar{\partial}$-Neumann problem and is the root of the modern theory of several complex variables. 
	
	The $\bar{\partial}$-Neumann problem has an unmanageable boundary condition. This boundary condition makes the equation not elliptic. Consequently, the regularities of the solution become obscure. The complete answer (i.e., the necessary and sufficient condition for regularities) is still lacking nowadays since the 1960s. We are going to highlight some remarkable results of regularities in the following paragraphs.
	
	In the 1960s, Kohn  \cite{Ko63}, \cite{Ko64} shows the global regularities are always satisfied if the domain is assumed to be strongly pseudoconvex. Indeed, he showed the following Sobolev estimates holds: \[\|u\|_{W^{k+1}}\lesssim\|\Box u\|_{W^k}.\] That is, $u$  gains $1$ more derivative than $\Box u$. Later Kohn's estimate has been improved to the following subelliptic estimates by Catlin \cite{Ca87}, \cite{Ca83} and D'Angelo \cite{DA79}, \cite{DA82}: \[\|u\|_{W^{k+\epsilon}}\lesssim\|\Box u\|_{W^k},\] where $\epsilon>0$ is a small positive number. This estimate holds for bounded pseudoconvex domains of finite type. The pseudoconvex domains of finite type is a generalization of the strongly pseudoconvex domains.
	
	During that period, people believed the global regularities hold for all bounded pseudoconvex domains. Hence, a study of pseudoconvex domains of infinite type is necessary. In the 1990s, Boas--Straube \cite{BS90}, \cite{BS91}, \cite{BS91b}, \cite{BS93} has found a series of conditions for domains of infinite type (see Harrington \cite{Ha11}, Straube \cite{St10} as well). Under these conditions, the global regularities are satisfied (see Chen \cite{Ch91} as well). On the other hand, almost at the same time, Barrett \cite{Ba92} showed the Sobolev estimates does not hold for a  family of bounded pseudoconvex domains. Later, Christ \cite{Ch96} showed on these domains, the inverse of $\Box$ fails to preserve the smoothness. These domains are known as $\beta$-worm domains constructed by Diederich--Forn\ae ss \cite{DF77a} and clearly they do not satisfy the conditions of Boas--Straube. This discovery breaks the beliefs about the global regularities of the $\bar{\partial}$-Neumann problem.
	
	From another point of view, the ``Diederich--Forn\ae ss index'' has a root in pluripotential theory. The original motivation was to seek for a bounded plurisubharmonic function on each bounded pseudoconvex domains. In 1977, Diederich--Forn\ae ss \cite{DF77b} were able to construct the bounded plurisubharmonic functions. Indeed, they proved for all bounded pseudoconvex domains with smooth boundary, there exists a $\tau_\rho\in(0,1]$ and a defining function $\rho$ so that $-(-\rho)^{\eta_\rho}$ is bounded and plurisubharmonic. Diederich--Forn\ae ss \cite{DF77a} also showed the $\tau_\rho$ may vary for different domains. For example, $\tau_\rho$ could be equal to 1 on strongly pseudoconvex domains while it has to be close to 0 on $\beta$-worm domains. This leads a definition of  Diederich--Forn\ae ss index (see Chen--Fu \cite{CF11}). Associated to each bounded pseudoconvex domain with smooth boundary, the real number 
	\[\eta:=\sup \tau_\rho ,\] where the supremum is taken over all defining functions of $\Omega$, is called the Diederich-Forn\ae ss index of the domain $\Omega$. A domain is said to be of trivial index if its Diederich--Forn\ae ss index is equal to 1. Otherwise, it is said to be of nontrivial index. The $\beta$-worm domains are of nontrivial index (see Diederich--Forn\ae ss \cite{DF77a} and Liu \cite{Li17}) and the strongly pseudoconvex domains are of trivial index. The Diederich--Forn\ae ss index has been intensely studied by Adachi \cite{Ad15}, Adachi--Brinkschulte \cite{AB14}, Harrington \cite{Ha17}, Abdulsahib--Harrington \cite{AH18}, Krantz--Liu--Peloso \cite{KLP16},  and Liu \cite{Li16}, \cite{Li17},\cite{Li17b}. 
	
	Looking at both global regularities and the Diederich--Forn\ae ss index, there are many pieces of evidence indicating the global regularities and trivial index happens simultaneously (compare \cite{BS91} \cite{BS93} with \cite{Li17} \cite{Li17b}), it is of interest to find a direct relation between the Diederich--Forn\ae ss index and the regularities. However, it turns out this problem is rather difficult. In 1999, Kohn \cite{Ko99} showed if a defining function satisfies certain properties of plurisubharmonicity, then the domain satisfies global regularities. His result was improved by Pinton--Zampieri \cite{PZ14} in 2014. In a similar fashion, Harrington \cite{Ha11} was able to give a concrete condition under which the global regularities are satisfied. All these conditions rely on the existence of certain special defining functions. In other words, these conditions give some indirect connection between the index and the regularities. After all, the Diederich--Forn\ae ss index is independent of defining functions. In contrast to this, Berndtsson--Charpentier \cite{BC00} in 2000 gave a direct relation. They showed that the Sobolev estimates of order $s>0$ holds if the Diederich--Forn\ae ss index is $2s$. Their result is very useful for a Sobolev estimate of the order up to $\frac{1}{2}$.  However, this regularity estimate is far from the global regularity estimates (where the order goes to infinity). Hence, a result directly connecting the trivial index and the global regularities is necessary.
	
	In this paper, we connect the trivial index and the global regularities under a mild assumption. As far as we know, it is the first time that a direct connection between the trivial index and global regularities is discovered. 
	
	Since it is classical that strongly pseudoconvex domains are both of trivial index and of global regularity, our task focuses on weakly pseudoconvex domains. Our main theorem is as follows.
	
	\begin{theorem}\label{mainthm}
		Let $\Omega$ be a pseudoconvex domain in $\mathbb{C}^n$ with smooth boundary and satisfy maximal estimates. Suppose the set $\Sigma$ of weakly pseudoconvex points is contained in a submanifold (of $\partial\Omega$) which is perpendicular to $L_n-\overline{L}_n$. If $\Omega$ is of the trivial index, then the inverse operator $N_q$ of $\Box_q$  ($q>1$) maps $W^k$ continuously to $W^k$ for all non-negative integer $k\geq 0$.
	\end{theorem}
	
	In this paper, we cannot remove the condition ``the set $\Sigma$ is contained in a submanifold perpendicular to $L_n-\overline{L}_n$'' due to the technique we have used in this paper. However, we find there is a large class of bounded pseudoconvex domains satisfying our condition including the $\beta$-worm domains.
	
	It is well-known that all the Sobolev estimates imply the global regularities and the condition R. As a corollary, we obtain the following.
	
	\begin{corollary}\label{maincor}
		Let $\Omega$ be a pseudoconvex domain in $\mathbb{C}^n$ with smooth boundary and satisfy maximal estimates. Suppose the set $\Sigma$ of weakly pseudoconvex points is contained in a submanifold (of $\partial\Omega$) which is perpendicular to $L_n-\overline{L}_n$. If $\Omega$ is of the trivial index,  then the $\Omega$ satisfies global regularity and condition R.
	\end{corollary}

	Theorem \ref{mainthm} and Corollary \ref{maincor} may be used to determine the global regularities for domains by its Diederich--Forn\ae ss index. This gives a new point of view for the globally regular domains.
	
	We also note that in $\mathbb{C}^2$, all bounded pseudoconvex domain with smooth boundary satisfies a maximal estimate (see Lemma 4.32 in Straube \cite{St10}). Thus, in the 2-dimensional case, the assumption of maximal estimates can be dropped and our theorem can be simplified as follows.
	\begin{theorem}\label{2ndthm}
		Let $\Omega$ be a pseudoconvex domain in $\mathbb{C}^2$ with smooth boundary. Suppose the set $\Sigma$ of weakly pseudoconvex points is contained in a submanifold (of $\partial\Omega$) which is perpendicular to $L_n-\overline{L}_n$. If $\Omega$ is of the trivial index, then the $\bar{\partial}$-Neumann operator is globally regular and satisfies condition R.
	\end{theorem}
	
	Assume $\Sigma$ is a real curve. It is known the Diederich--Forn\ae ss index is equal to 1 for the case of $\mathbb{C}^2$ (see Theorem 4.5 of Liu \cite{Li17b}). Combining this fact with Theorem \ref{2ndthm}, one gets the following corollary.
	
	\begin{corollary}
		Let $\Omega$ be a pseudoconvex domain in $\mathbb{C}^2$ with smooth boundary. Suppose the set $\Sigma$ of weakly pseudoconvex points is a real curve which is perpendicular to $L_n-\overline{L}_n$. Then the $\bar{\partial}$-Neumann operator is globally regular and satisfies condition R.
	\end{corollary}
	
	After preliminaries, we use Section \ref{thedf}-\ref{commutatorssection} to prepare the proof of Theorem \ref{mainthm}. Eventually, the proof of Theorem \ref{mainthm} is given in Section \ref{proof}.
	
	\section{Preliminary}\label{preliminary}
	Let $\Omega$ be a bounded pseudoconvex domain with smooth boundary in $\mathbb{C}^n$. Let $\Sigma$ denote the set of weakly pseudoconvex points. Let $\lbrace U_\alpha\rbrace_{\alpha=1}^\mathcal{M}$ be finite open sets of which the union covers $\Sigma$. For each $U_\alpha$, let $\lbrace \sqrt{2}L_j\rbrace_{j=1}^{n-1}$ be unit $(1, 0)$ tangential vector fields with smooth coefficients in $U_\alpha$ such that $\lbrace L_j|_{\partial\Omega}\rbrace_{j=1}^{n-1}$ are tangential to $\partial\Omega$. Let $L_n=(\sum|\frac{\partial\delta}{\partial z_i}|^2)^{-1}\sum \frac{\partial\delta}{\partial \bar{z}_i}\frac{\partial}{\partial z_i}$ be the $(1,0)$ complex normal vector, where $\delta$ denotes the signed distance function for $\Omega$. We can see that $L_n$ is globally defined which is independent from the coordinate charts. Let $u$ be a $(0, q)$-form for $1\leq q\leq n$. Let ${\sum_I}'$ denote the sum over all $q$-tuple $I$ with increasing indices. We have the following identities (see Harrington--Liu \cite{HL19}):
	
	\[u={\sum_{I}}'f_I\omega_{\overline{L}_I}=\frac{1}{2^q}{\sum_{I}}'g(u,\omega_{\overline{L}_I})\omega_{\overline{L}_I}\]
	
	\[\bar{\partial}u=\frac{1}{2^q}{\sum_{I}}' \sum_{i\notin I} g(\nabla_{\overline{L}_i}u,\omega_{\overline{L}_I})\omega_{\overline{L}_i}\wedge\omega_{\overline{L}_I},\]
	and
	\[\bar{\partial}^* u=-\frac{1}{2^{q-1}}{\sum_{I}}'\sum_{i\in I} g(\nabla_{L_i}u, \omega_{\overline{L}_I})\overline{L}_i\lrcorner\omega_{\overline{L}_I}.\]
	
	Here, $g$ is the standard Euclidean metric on $\mathbb{C}^n$.
	
	We observe that $u\in\Dom(\bar{\partial}^*)$ if and only if $g(u, \omega_{\overline{L}_J}\wedge\omega_{\overline{L}_n})=0$ on $\partial\Omega$, for all arbitrary $q-1$ tuple $J$ with increasing indexes.

	Moreover, we also want to remark that the following identities are useful:
	
	\[\nabla_{L_i}\omega_{\overline{L}_j}=2\sum_{k}g(\nabla_{L_i}L_j, L_k)\omega_{\overline{L}_k}\qquad\text{and}\qquad\nabla_{\overline{L}_i}\omega_{\overline{L}_j}=2\sum_{k}g(\nabla_{\overline{L}_i}L_j, L_k)\omega_{\overline{L}_k}.\]

	This paper is mainly on Sobolev estimates, i.e., the estimates on $W^k$ norms. The Sobolev norm of order $k\in\mathbb{N}$ is defined to be \[\|u\|^2_{W^k}=\sum_{j=0}^{k}\int_{\Omega}g(\nabla^ju,\nabla^ju)\,dV.\] We also denote the $L^2$ norm by $\|\cdot\|=\langle\cdot, \cdot\rangle$. Because of the definition, one may want to consider $\nabla u$ or $\nabla_{L_n-\overline{L}_n} u$. However, $\nabla_{L_n-\overline{L}_n} u$ is not in $\Dom(\bar{\partial}^*)$ and the Morrey--Kohn--H\"{o}rmander estimate does not apply. Hence, we use the classical method (see Straube \cite{St10} and Chen--Shaw \cite{CS01}) as follows. We define the following notation $\nabla^c$. 
	
	\begin{definition}
		Let $u=\frac{1}{2^q}\sum_{I}g(u,\omega_{\overline{L}_I})\omega_{\overline{L}_I}$. Then \[\nabla^c_{L_n-\overline{L}_n} u:=\frac{1}{2^q}\sum_{I}(L_n-\overline{L}_n)g(u,\omega_{\overline{L}_I})\omega_{\overline{L}_I}.\]
	\end{definition}
	In other words, $\nabla^c$ only acts on the coefficients of $\omega_{\overline{L}_I}$. One can see that $\nabla_{L_n-\overline{L}_n}^c u\in\Dom(\bar{\partial}^*)$ whenever $u\in\Dom(\bar{\partial}^*)$.

	For the following, we recall the maximal estimates of Derridj \cite{De78} (see also Straube \cite{St10}, \c{C}elik--\c{S}ahuto\u{g}lu--Straube \cite{CSS19} and Derridj \cite{De17}).
	
	We define the maximal estimate as follows.
	
	\begin{definition}
		Let $\Omega$ be a bounded pseudoconvex domain in $\mathbb{C}^n$ with smooth boundary and $p\in\partial\Omega$. Let $u\in\Dom(\bar{\partial}^*)$ supported in a neighborhood $V_p$ of $p\in\partial\Omega$. The following estimates:
		\[\sum_{i=1}^{n-1}\|\nabla_{L_i}u\|^2+\sum_{j=1}^n\|\nabla_{\overline{L}_i}u\|^2\lesssim \|\bar{\partial}u\|^2+\|\bar{\partial}^*u\|^2+\|u\|^2\] is referred to a maximal estimate at $p$.
	\end{definition}
	
	One can immediately obtain the following inequality which we will use further
	\[\sum_{i=1}^{n-1}\|\nabla_{L_i}u\|^2+\sum_{j=1}^n\|\nabla_{\overline{L}_i}u\|^2\lesssim \|\bar{\partial}u\|^2+\|\bar{\partial}^*u\|^2.\] This is because on bounded pseudoconvex domain with a smooth boundary $\|u\|^2\lesssim\|\bar{\partial}u\|^2+\|\bar{\partial}^*u\|^2$ by the basic estimate of Morrey--Kohn--H\"{o}rmander.
	
	Let us recall the necessary and sufficient condition of the maximal estimates given by Derridj \cite{De78}.
	
	\begin{theorem}[Derridj]
		Let $\Omega$ be a bounded pseudoconvex domain in $\mathbb{C}^n$ with a smooth boundary and $p\in\partial\Omega$. The maximal estimate holds near $p$ if and only if the eigenvalues of the Levi form are comparable in a neighborhood of $p$, $V_p$.
	\end{theorem}
	
	Let $\Omega$ be a bounded pseudoconvex domain with smooth boundary which satisfies the maximal estimates. If a point is in the set $\Sigma$ of weakly pseudoconvex points, then this point is a weakly pseudoconvex point in all tangential $(1,0)$ direction. This is because all eigenvalues should be comparable (one is zero and others have to be zero as well).
	
	For the definition of the condition R, see Chen--Shaw \cite{CS01}. By Bell--Ligocka \cite{BL80}, the condition R is equivalent to the following condition.
	
	\begin{theorem}[Bell--Ligocka]
		For every $s \geq 0$ there exists $M \geq 0$ such that the Bergman projector $P$ is bounded as an operator from $W^{s+M}(\Omega)$ to $H^{s}(\Omega)$.
	\end{theorem}
	It is known that the global regularities implies the condition R.
	
	Finally, we want to clarify the following: in this paper, the notation \[\|u\|^2_k=\int_{\Omega}g(\nabla^k u,\nabla^k u)\,dV\] is slightly different from the norm for the Sobolev space of order $k$. The latter is defined to be \[\|u\|^2_{W^k}=\sum_{j=0}^{k}\int_{\Omega}g(\nabla^ju,\nabla^ju)\,dV.\]
	
	At the last, we want to remind the reader with the integration by parts. If $X$ is a $(1,0)$ vector field so that $X$ is tangential to $\partial\Omega$, then we have the following formula of integration by parts: 
	\[\int_{\Omega}Xf+\divergence X\,dV=\int_{\Omega}Xf+2\sum_{i=1}^{n}g(\nabla_{L_i}X,L_i)\,dV=0.\]
	\section{The Diederich--Forn\ae ss index}\label{thedf}
	In \cite{Li17}, Liu gave an equivalent definition of the Diederich--Forn\ae ss index. The equivalent definition gives information of the boundary of trivial index. We remind the reader with this equivalence.
	\begin{theorem}[Liu \cite{Li17}]\label{pre}
		Let $\Omega$ be a bounded pseudoconvex domain with smooth boundary in $\mathbb{C}^n$. Let $L$ be an arbitrary smooth $(1,0)$ tangent vector field on $\partial\Omega$. Let \[\Sigma_L:=\lbrace p\in\partial\Omega: \Hessian_\delta(L, L)=0 \enskip\text{at}\enskip p\rbrace.\] The Diederich--Forn\ae ss index for $\Omega$ is 1 if and only if for any $\sigma>0$, there exists a smooth function $\phi$ which is defined on a neighborhood of $\Sigma$ in $\partial\Omega$, so that on $\Sigma_L$,
		\[
		\sigma|\overline{L}\phi-2g(\nabla_{L_n}L_n, L)|^2-\Hessian_\phi(L, L)+2g(\nabla_L\nabla_{L_n}L_n, L)\leq 0,
		\] for all $L$.
	\end{theorem}

	%
	%
	In view of this theorem, the bound $0$ on the right hand side is difficult for further analysis. Hence, we find a practical bound by the following lemma. Note also if we assume the domain satisfies the maximal estimates, $\Sigma_L=\Sigma$ for all $L$. 
	
	\begin{lemma}
		Let $\Omega$ be a bounded pseudoconvex domain with smooth boundary of trivial index in $\mathbb{C}^n$. Asume $\Omega$ satisfies the maximal estimates. Let $L$ be an arbitrary smooth $(1,0)$ tangent vector field on $\partial\Omega$. Then for arbitrary $\sigma>1$, there exists a smooth function $\tilde{\phi}$ which is defined on a neighborhood of $\Sigma$ in $\partial\Omega$ and a number $\mu=\mu_{\sigma}>0$, so that on $\Sigma$,
		\[
		\frac{\mu}{4}+(\sigma-1)\left|\overline{L}\tilde{\phi}-2g(\nabla_{L_n}L_n, L)\right|^2-\Hessian_{\tilde{\phi}}(L, L)+2g(\nabla_L\nabla_{L_n}L_n, L)\leq 0,
		\] for all $L$.
	\end{lemma}
	\begin{proof}
		Let $\tilde{\phi}=\phi+\mu|z|^2$, where $\mu>0$ will be determined later. By the assumption, we obtain that
		\begin{align*}
		0\geq &\sigma |L\phi-2g(\nabla_LL_n,L_n)|^2-\Hessian_\phi(L,L)+2g(\nabla_L\nabla_{L_n}L_n,L)\\
		=&\sigma |-L(\tilde{\phi}+\mu|z|^2)+2g(\nabla_LL_n,L_n)|^2-\Hessian_{\tilde{\phi}}(L,L)+2g(\nabla_L\nabla_{L_n}L_n,L)+\frac{\mu}{2}\\
		=&\sigma |-L\tilde{\phi}+2g(\nabla_LL_n,L_n)|^2+\sigma\mu^2|L|z|^2|^2-2\sigma\mu\Re(-L\tilde{\phi}+2g(\nabla_LL_n,L_n))\overline{L}|z|^2\\
		&-\Hessian_{\tilde{\phi}}(L,L)+2g(\nabla_L\nabla_{L_n}L_n,L)+\frac{\mu}{2}\\
		\geq&\sigma |-L\tilde{\phi}+2g(\nabla_LL_n,L_n)|^2+\sigma\mu^2|L|z|^2|^2-\sigma^2\mu^2|\overline{L}|z|^2|^2-|-L\tilde{\phi}+2g(\nabla_LL_n,L_n)|^2\\
		&-\Hessian_{\tilde{\phi}}(L,L)+2g(\nabla_L\nabla_{L_n}L_n,L)+\frac{\mu}{2}\\
		=&(\sigma-1)|-L\tilde{\phi}+2g(\nabla_LL_n,L_n)|^2+(\sigma-\sigma^2)\mu^2|L|z|^2|^2-\Hessian_{\tilde{\phi}}(L,L)+2g(\nabla_L\nabla_{L_n}L_n,L)+\frac{\mu}{2}
		\end{align*}
		By shirking $\mu$, $(\sigma-\sigma^2)\mu^2|L|z|^2|^2\geq -\frac{\mu}{4}$ can be achieved. The reader is referred to the proof of Lemma 2.6 in Liu \cite{Li17}, where the same method was applied.
		Thus, we obtain that \[(\sigma-1)|-L\tilde{\phi}+2g(\nabla_LL_n,L_n)|^2-\Hessian_{\tilde{\phi}}(L,L)+2g(\nabla_L\nabla_{L_n}L_n,L)+\frac{\mu}{4}\leq 0.\]
	\end{proof}
	
	Observe that if we assume the subset $\Sigma$ of weakly pseudoconvex points in $\partial\Omega$ is contained in a real submanifold perpendicular to $L_n-\overline{L}_n$, we can find a specific extension $\psi$ which is defined in a neighborhood of $\Sigma$ in $\partial\Omega$ so that $\psi=\tilde{\phi}$ on $\Sigma$ and $(L_n-\overline{L}_n)\psi=0$. This $\psi$ only differs from $\tilde{\phi}$ in the $L_n-\overline{L}_n$ direction. Consider that 
	\[
	\Hessian_{\tilde{\phi}}(L, L)=L\overline{L}\tilde{\phi}-\nabla_{L}\overline{L}\tilde{\phi}=L\overline{L}\tilde{\phi}-2\sum_{i=1}^{n}g(\nabla_L\overline{L}, \overline{L}_i)\overline{L}_i\tilde{\phi}.
	\]
	We find, due to the fact that the Levi form $g(\nabla_L\nabla\delta,L)=0$, the term \[g(\nabla_L\overline{L}, \overline{L}_n)\overline{L}_n\tilde{\phi}=-g(\nabla_{L}L_n, L)\overline{L}_n\tilde{\phi}=0.\]
	So on $\Sigma$, \[
	\Hessian_{\tilde{\phi}}(L, L)=L\overline{L}\tilde{\phi}-2\sum_{i=1}^{n-1}g(\nabla_L\overline{L}, \overline{L}_i)\overline{L}_i\tilde{\phi}=\Hessian_\psi(L, L).
	\]
	It can be seen that $L\tilde{\phi}=L\psi$ on $\Sigma$ because $\tilde{\phi}=\psi$ for all directions perpendicular to $L_n-\overline{L}_n$. Thus, we have the following lemma.
	
	\begin{lemma}\label{properties}
		Let $\Omega$ be a bounded pseudoconvex domain with smooth boundary of trivial index in $\mathbb{C}^n$. Suppose $\Omega$ satisfies a maximal estimate and the set $\Sigma$ of weakly pseudoconvex points is contained in a real submanifold in $\Omega$ which is perpendicular to $L_n-\overline{L}_n$. Let $L$ be an arbitrary smooth $(1,0)$ tangent vector field on $\partial\Omega$. Then for arbitrary $\sigma>1$, there exists a smooth function $\psi$ which is defined on a neighborhood of $\Sigma$ in $\partial\Omega$ and a number $\mu=\mu_{\sigma}>0$, so that on $\Sigma$ the following two properties are satisfied,
		\begin{enumerate}
			\item \[(L_n-\overline{L}_n)\psi=0\] and
			\item     \[
			\frac{\mu}{4}+(\sigma-1)\left|\overline{L}\psi-2g(\nabla_{L_n}L_n, L)\right|^2-\Hessian_\psi(L, L)+2g(\nabla_L\nabla_{L_n}L_n, L)\leq 0,
			\] for all $L$.
		\end{enumerate}
	\end{lemma}
	
	\begin{proof}
		Since $\Omega$ satisfies a maximal estimate, the eigenvalues of Levi form are comparable at all points in $\Sigma$. This implies that if  $\Hessian_\delta(L, L)=0$ at $p$ for one direction $L_p\in T^{(1, 0)}\partial\Omega$, then $\Hessian_\delta(L, L)=0$ at $p$ for all $(1,0)$ tangential directions. Thus, $\Sigma=\Sigma_L$ for all $L$. Since $\Sigma$ is contained in a (sub)mainfold in $\partial\Omega$ which is perpendicular to $L_n-\overline{L}_n$, we can use the argument before the statement of Lemma to complete the proof.
	\end{proof}
	
	Now, we extend $\psi$ to a neighborhood of $\partial\Omega$ to a real function so that $(L_n+\overline{L}_n)\psi=4g(\nabla_{L_n}L_n, L_n)$ on $\partial\Omega$. In particular, $g(\nabla_{L_n}L_n, L_n)=g(\nabla_{L_n}\nabla\delta, L_n)$ is a real function. By continuity, we can see that 
	\[2\Re(\sigma-1)|L\psi-2g(\nabla_LL_n,L_n)|^2-2\Re\Hessian_\psi(L,L)+2\Re g(\nabla_L\nabla_{L_n}L_n,L)+\frac{\mu}{10}\leq 0\] in a tubular neighborhood of $\Sigma$ in $\mathbb{C}^n$. The reason why we mark the inequality with real part is because $g(\nabla_L\nabla_{L_n}L_n, L)$, in general, is not a real function outside $\Sigma$.
	
	From the construction on $\Sigma$, we have $(L_n-\overline{L}_n)\psi=0$ and $(L_n+\overline{L}_n)\psi=4g(\nabla_{L_n}L_n, L_n)$. This gives $L_n\psi=\overline{L}_n\psi=2g(\nabla_{L_n}L_n,L_n)$ on $\Sigma$. We will use this fact later.

	\section{The properties of maximal estimates}
	
	We have already obtained a real function $\psi$ defined in a neighborhood of $\Sigma$ in $\mathbb{C}^n$ in the previous section. In this section, we will use the maximal estimates to obtain an estimate. This estimate will be combined with a method of Chen--Shaw \cite{CS01} (See Harrington--Liu \cite{HL19} as well) to prove Theorem \ref{mainthm}.  From now on, we assume $\Sigma$ is covered by only one coordinate chart, i.e. $\mathcal{M}=1$. We will relax this assumption later (see Harrington--Liu \cite{HL19}). 
	
	Let $T=e^\psi(L_n-\overline{L}_n)$ and $w=(\nabla_T^c)^ku$ for some $k\in\mathbb{N}$, where $u\in\Dom(\Box)\cap C_{(0,q)}^\infty(\overline{\Omega})$, for $1\leq q\leq n$. Since $\Omega$ satisfies the maximal estimate and $w\in \Dom(\bar{\partial}^*)$, we have that
	\[\sum_{i=1}^{n-1}\|\nabla_{L_i}w\|^2\leq C(\|\bar{\partial}w\|^2+\|\bar{\partial}^* w\|^2).\] 
	
	Let $\lambda$ has a support in a neighborhood of $\Sigma$ in $\mathbb{C}^n$. The support will be chosen later. We also assume $\lambda\in [0,1]$ and $\lambda\equiv 1$ on $\Sigma$. Moreover, since $\lambda w$ is also in the $\Dom(\bar{\partial}^*)$, we have the following:
	\[\sum_{i=1}^{n-1}\|\nabla_{L_i}\lambda w\|^2\leq C(\|\bar{\partial}\lambda w\|^2+\|\bar{\partial}^* \lambda w\|^2).\]
	
	Let $L$ be a $(1,0)$ vector fields in a tubular neighborhood of $\partial\Omega$ which is tangential to $\partial\Omega$. Observe that, by the second property of Lemma \ref{properties},
	\begin{align*}
	0\geq&\int_{\Omega}2\lambda^2 g(w,w)(\sigma-1)|-L\psi+2g(\nabla_LL_n,L_n)|^2\,dV\\
	&-\int_{\Omega}\lambda^2 g(w,w)(2\Hessian_\psi(L,L)-2g(\nabla_L\nabla_{L_n}L_n,L)-2g(L, \nabla_L\nabla_{L_n}L_n)-\frac{\mu}{10})\,dV\\
	=&\int_{\Omega}2(\sigma-1) \lambda^2 g(w,w)|\overline{L}\psi-2g(\nabla_{L_n}L_n,L)|^2\,dV+\int_{\Omega}\frac{\mu}{10}\lambda^2e^{2\psi}g(v,v)\,dV\\&-2\int_{\Omega}\lambda^2g(w,w)(\Hessian_\psi(L,L)-2g(\nabla_L\nabla_{L_n}L_n,L))\,dV\\
	&-2\int_{\Omega}\lambda^2g(w,w)(\Hessian_\psi(L,L)-2g(L,\nabla_L\nabla_{L_n}L_n))\,dV
	\end{align*}
	
	We focus on the second term as follows,
	\begin{align*}
	&2\int_{\Omega}\lambda^2g(w,w)(-\Hessian_\psi(L,L)+2g(\nabla_L\nabla_{L_n}L_n,L))\,dV\\
	=&2\int_{\Omega}\lambda^2g(w,w)(-L\overline{L}\psi+2Lg(\nabla_{L_n}L_n,L)+\nabla_L\overline{L}\psi-2g(\nabla_{L_n}L_n,\nabla_{\overline{L}}L))\,dV\\
	=&2\int_{\Omega}\lambda^2g(w,w)(-L\overline{L}\psi+2Lg(\nabla_{L_n}L_n,L))-\lambda^2g(w,w)(-\nabla_L\overline{L}\psi+2g(\nabla_{L_n}L_n,\nabla_{\overline{L}}L))\,dV\\
	=&-2\int_{\Omega}L(\lambda^2g(w,w))(-\overline{L}\psi+2g(\nabla_{L_n}L_n,L))+2\sum_{i=1}^{n}g(\nabla_{L_i}L, L_i)\lambda^2g(w,w)(-\overline{L}\psi+2g(\nabla_{L_n}L_n,L))\,dV\\
	&-2\int_{\Omega}\lambda^2g(w,w)(-\nabla_L\overline{L}\psi+2g(\nabla_{L_n}L_n,\nabla_{\overline{L}}L))\,dV\\
	=&-2\int_{\Omega}L(\lambda^2g(w,w))(-\overline{L}\psi+2g(\nabla_{L_n}L_n,L))+2\sum_{i=1}^{n}g(\nabla_{L_i}L, L_i)\lambda^2g(w,w)(-\overline{L}\psi+2g(\nabla_{L_n}L_n,L))\,dV\\&+4\int_{\Omega}\lambda^2g(w,w)\sum_{i=1}^ng(\nabla_L\overline{L},\overline{L}_i)(\overline{L}_i\psi-2g(\nabla_{L_n}L_n, L_i))\,dV.
	\end{align*}
	
	Thus, 
	\begin{align*}
	&0\\\geq&\int_{\Omega}2(\sigma-1) \lambda^2 g(w,w)|-\overline{L}\psi+2g(\nabla_{L_n}L_n,L)|^2\,dV+\int_{\Omega}\lambda^2\frac{\mu}{10}g(w,w)\,dV\\
	&-2\int_{\Omega}L(\lambda^2g(w,w))(-\overline{L}\psi+2g(\nabla_{L_n}L_n,L))+2\sum_{i=1}^{n}g(\nabla_{L_i}L, L_i)\lambda^2g(w,w)(-\overline{L}\psi+2g(\nabla_{L_n}L_n,L))\,dV\\&+4\int_{\Omega}\lambda^2g(w,w)\sum_{i=1}^ng(\nabla_L\overline{L},\overline{L}_i)(\overline{L}_i\psi-2g(\nabla_{L_n}L_n, L_i))\,dV\\
	&-2\int_{\Omega}\overline{L}(\lambda^2g(w,w))(-L\psi+2g(L,\nabla_{L_n}L_n))+2\sum_{i=1}^{n}g( L_i,\nabla_{L_i}L)\lambda^2g(w,w)(-L\psi+2g(L,\nabla_{L_n}L_n))\,dV\\&+4\int_{\Omega}\lambda^2g(w,w)\sum_{i=1}^ng(\overline{L}_i,\nabla_L\overline{L})(L_i\psi-2g(L_i,\nabla_{L_n}L_n))\,dV.
	\end{align*}
	
	Now we let $L=L_j$ and sum over $1\leq j\leq n-1$. We obtain that 
	
	\begin{align*}
	0\geq&\sum_{j=1}^{n-1}\int_{\Omega}2(\sigma-1) \lambda^2 g(w,w)|-\overline{L}_j\psi+2g(\nabla_{L_n}L_n,L_j)|^2\,dV+(n-1)\int_{\Omega}\lambda^2\frac{\mu}{10}g(w,w)\,dV\\
	&+4\sum_{j=1}^{n-1}\Re\int_{\Omega}L_j(\lambda^2g(w,w))(\overline{L}_j\psi-2g(\nabla_{L_n}L_n,L_j))\,dV\\&+8\sum_{j=1}^{n-1}\sum_{i=1}^{n}\Re\int_{\Omega}g(\nabla_{L_i}L_j, L_i)\lambda^2e^{2\psi}g(v,v)(\overline{L}_j\psi-2g(\nabla_{L_n}L_n,L_j))\,dV\\&+8\sum_{j=1}^{n-1}\sum_{i=1}^n\Re\int_{\Omega}\lambda^2g(w,w)g(\nabla_{L_j}\overline{L}_j,\overline{L}_i)(\overline{L}_i\psi-2g(\nabla_{L_n}L_n, L_i))\,dV\\
	=&\sum_{j=1}^{n-1}\int_{\Omega}2(\sigma-1) \lambda^2 g(w,w)|-\overline{L}_j\psi+2g(\nabla_{L_n}L_n,L_j)|^2\,dV+(n-1)\int_{\Omega}\lambda^2\frac{\mu}{10}g(w,w)\,dV\\
	&+4\sum_{j=1}^{n-1}\Re\int_{\Omega}L_j(\lambda^2g(w,w))(\overline{L}_j\psi-2g(\nabla_{L_n}L_n,L_j))\,dV\\&+8\sum_{j=1}^{n-1}\sum_{i=1}^{n}\Re\int_{\Omega}g(\nabla_{L_i}L_j, L_i)\lambda^2g(w,w)(\overline{L}_j\psi-2g(\nabla_{L_n}L_n,L_j))\,dV\\&-8\sum_{j=1}^{n-1}\sum_{i=1}^n\Re\int_{\Omega}\lambda^2g(w,w)g(\nabla_{L_j}L_i, L_j)(\overline{L}_i\psi-2g(\nabla_{L_n}L_n, L_i))\,dV\\
	=&\sum_{j=1}^{n-1}2\int_{\Omega}(\sigma-1) \lambda^2 g(w,w)|-\overline{L}_j\psi+2g(\nabla_{L_n}L_n,L_j)|^2\,dV+(n-1)\int_{\Omega}\lambda^2\frac{\mu}{10}g(w,w)\,dV\\
	&+4\sum_{j=1}^{n-1}\Re\int_{\Omega}L_j(\lambda^2g(w,w))(\overline{L}_j\psi-2g(\nabla_{L_n}L_n,L_j))\,dV\\&+8\sum_{j=1}^{n-1}\Re\int_{\Omega}g(\nabla_{L_n}L_j, L_n)\lambda^2g(w,w)(\overline{L}_j\psi-2g(\nabla_{L_n}L_n,L_j))\,dV\\&-8\sum_{j=1}^{n-1}\Re\int_{\Omega}\lambda^2g(w,w)g(\nabla_{L_j}L_n, L_j)(\overline{L}_n\psi-2g(\nabla_{L_n}L_n, L_n))\,dV.
	\end{align*} 
	
	We let \[M=\sqrt{2}\max|g(\nabla_{L_n}L,L_n)|+1,\] the maximum is taken for all unit $(1,0)$ vector fields $L$ which is tangential to $\partial\Omega$ and all points in a neighborhood of $\partial{\Omega}$.  Clearly, $M$ is independent of $\psi$. We let $\lambda$ has compact support in a neighborhood of $\Sigma$. Shrink $\support\lambda$ if necessary so that \[\sum_{j=1}^{n-1}|g(\nabla_{L_j}L_n, L_j)(\overline{L}_n\psi-2g(\nabla_{L_n}L_n, L_n))|\leq \frac{\mu}{10000}(n-1)\] on $\support \lambda$. This can be done because $g(\nabla_{L_j}L_n, L_j)=0$ on $\Sigma$. (Of course, the choice of $\lambda$ depends on $\psi$.)

	Thus, we obtain 
	\begin{align*}
	&\sum_{j=1}^{n-1}2\int_{\Omega}(\sigma-1)\lambda^2 g(w,w)|\overline{L}_j\psi-2g(\nabla_{L_n}L_n,L_j)|^2\,dV+\frac{\mu}{10}(n-1)\int_{\Omega}\lambda^2g(w,w)\,dV\\
	\leq&4\sum_{j=1}^{n-1}\int_{\Omega}\left|L_j(\lambda^2g(w,w))\right|\left|\overline{L}_j\psi-2g(\nabla_{L_n}L_n,L_j)\right|\,dV+\frac{\mu}{1000}(n-1)\int_{\Omega}\lambda^2g(w,w)\,dV\\&+4\sum_{j=1}^{n-1}\int_{\Omega}\lambda^2g(w,w)\left(1+M^2\left|\overline{L}_j\psi-2g(\nabla_{L_n}L_n,L_j)\right|^2\right)\,dV\\
	\leq&4\sum_{j=1}^{n-1}\int_{\Omega}\left|g(\nabla_{L_j}\lambda w,\lambda w)\right|\left|\overline{L}_j\psi-2g(\nabla_{L_n}L_n,L_j)\right|\,dV+\frac{\mu}{1000}(n-1)\int_{\Omega}\lambda^2g(w,w)\,dV\\&
	+4\sum_{j=1}^{n-1}\int_{\Omega}\left|g(\lambda w,\nabla_{\overline{L}_j}\lambda w)\right|\left|\overline{L}_j\psi-2g(\nabla_{L_n}L_n,L_j)\right|\,dV\\&+4\sum_{j=1}^{n-1}\int_{\Omega}\lambda^2g(w,w)\left(1+M^2\left|\overline{L}_j\psi-2g(\nabla_{L_n}L_n,L_j)\right|^2\right)\,dV\\
	\leq&2\sum_{j=1}^{n-1}\|\nabla_{L_j}\lambda w\|^2+2\sum_{j=1}^{n-1}\|\nabla_{\overline{L}_j}\lambda w\|^2+4\sum_{j=1}^{n-1}\int_{\Omega}g(\lambda w,\lambda w)\left|\overline{L}_j\psi-2g(\nabla_{L_n}L_n,L_j)\right|^2\,dV\\&+\frac{\mu}{1000}(n-1)\int_{\Omega}\lambda^2g(w,w)\,dV+4\sum_{j=1}^{n-1}\int_{\Omega}g(\lambda w,\lambda w)\left(1+M^2\left|\overline{L}_j\psi-2g(\nabla_{L_n}L_n,L_j)\right|^2\right)\,dV.
	\end{align*}
	
	That is,
	
	\begin{align*}
	&\sum_{j=1}^{n-1}2\int_{\Omega}(\sigma-1)\lambda^2 g(w,w)|\overline{L}_j\psi-2g(\nabla_{L_n}L_n,L_j)|^2\,dV+\frac{\mu}{10}(n-1)\int_{\Omega}\lambda^2g(w,w)\,dV\\
	\leq&2\sum_{j=1}^{n-1}\|\nabla_{L_j}\lambda w\|^2+2\sum_{j=1}^{n-1}\|\nabla_{\overline{L}_j}\lambda w\|^2+4\sum_{j=1}^{n-1}\int_{\Omega}g(\lambda w,\lambda w)\left|\overline{L}_j\psi-2g(\nabla_{L_n}L_n,L_j)\right|^2\,dV\\&+\frac{\mu}{1000}(n-1)\int_{\Omega}\lambda^2g(w,w)\,dV+4\sum_{j=1}^{n-1}\int_{\Omega}\lambda^2g(w,w)\left(1+M^2\left|\overline{L}_j\psi-2g(\nabla_{L_n}L_n,L_j)\right|^2\right)\,dV.
	\end{align*}

	Thus, we have that
	\begin{align*}
	&\sum_{j=1}^{n-1}\int_{\Omega}2(\sigma-1) g(\lambda w,\lambda w)|\overline{L}_j\psi-2g(\nabla_{L_n}L_n,L)|^2\,dV+\frac{\mu}{10}(n-1)\int_{\Omega}g(\lambda w,\lambda w)\,dV\\
	\leq&2\sum_{j=1}^{n-1}\|\nabla_{L_j}\lambda w\|^2+2\sum_{j=1}^{n-1}\|\nabla_{\overline{L}_j}\lambda w\|^2+\int_{\Omega}(4+4M^2)|\overline{L}_j\psi-2g(\nabla_{L_n}L_n,L_j)|^2g(\lambda w, \lambda w)\,dV\\&+(4+\frac{\mu}{1000}(n-1))\int_{\Omega}g(\lambda w,\lambda w)\,dV\\
	\leq&C\|\bar{\partial}\lambda w\|^2+C\|\bar{\partial}^* \lambda w\|^2+(4+4M^2)\sum_{j=1}^{n-1}\int_{\Omega}|\overline{L}_j\psi-2g(\nabla_{L_n}L_n,L_j)|^2g(\lambda w, \lambda w)\,dV\\&+\frac{\mu}{1000}(n-1)\int_{\Omega}g(\lambda w,\lambda w)\,dV,
	\end{align*}
	
	which implies 
	\begin{equation}\label{predbar}
	\begin{split}
	&(2\sigma-6-4M^2)\sum_{j=1}^{n-1} \int_{\Omega}|\overline{L}_j\psi-2g(\nabla_{L_n}L_n,L_j)|^2g(\lambda w,\lambda w)\,dV+\frac{\mu}{20}(n-1)\int_{\Omega}g(\lambda w,\lambda w)\,dV\\
	\leq&C\|\bar{\partial}\lambda w\|^2+C\|\bar{\partial}^* \lambda w\|^2.
	\end{split}    
	\end{equation}
	
	For this inequality, we have used the $L^2$-estimate of $\bar{\partial}$ equation in sense of Morrey--Kohn--H\"{o}rmander:
	\begin{align*}
	\|\lambda w\|^2\lesssim\|\bar{\partial}\lambda w\|^2+\|\bar{\partial}^*\lambda w\|^2,
	\end{align*}
	for all $\lambda w\in\Dom(\bar{\partial}^*)$.
	
	We note that $C$ is independent of $\psi$.

	\section{Estimates for $\langle\bar{\partial}w,\bar{\partial}w\rangle+\langle\bar{\partial}^* w,\bar{\partial}^* w\rangle$}
	In this section, we mainly deal with the right hand side of (\ref{predbar}). The method is similar to the one used in Chen--Shaw \cite{CS01} (see Chen \cite{Ch91} and Harrington--Liu \cite{HL19} as well).

	Consider \begin{align*}
	&\langle\bar{\partial}\lambda(\nabla^c_T)^ku,\bar{\partial}\lambda(\nabla^c_T)^ku\rangle+\langle\bar{\partial}^* \lambda(\nabla^c_T)^ku,\bar{\partial}^* \lambda(\nabla^c_T)^ku\rangle\\
	=&\langle\lambda(\nabla^c_T)^k\bar{\partial}u,\bar{\partial}\lambda(\nabla^c_T)^ku\rangle+\langle[\bar{\partial}, \lambda(\nabla^c_T)^k]u,\bar{\partial}\lambda(\nabla^c_T)^ku\rangle\\&+\langle\lambda(\nabla^c_T)^k\bar{\partial}^* u,\bar{\partial}^* \lambda(\nabla^c_T)^ku\rangle+\langle[\bar{\partial}^*, \lambda(\nabla^c_T)^k] u,\bar{\partial}^* \lambda(\nabla^c_T)^ku\rangle\\
	=&\sum_{j=0}^{k-1}\int_{\Omega}\lambda Tg((\nabla^c_T)^{j}\bar{\partial}u,(\nabla_{T}^c)^{k-j-1}\bar{\partial}\lambda(\nabla^c_T)^ku)\,dV+\langle\bar{\partial}u, \lambda (\nabla^c_T)^k\bar{\partial}\lambda(\nabla^c_T)^ku\rangle\\&+\langle[\bar{\partial}, \lambda(\nabla^c_T)^k]u,\bar{\partial}\lambda(\nabla^c_T)^ku\rangle\\
	&+\sum_{j=0}^{k-1}\int_{\Omega}\lambda Tg((\nabla^c_T)^j\bar{\partial}^* u,(\nabla_T^c)^{k-j-1}\bar{\partial}^* \lambda(\nabla^c_T)^ku)\,dV+B\langle\bar{\partial}^* u,\lambda(\nabla^c_T)^k\bar{\partial}^*\lambda(\nabla^c_T)^ku\rangle\\&+\langle[\bar{\partial}^*,\lambda(\nabla^c_T)^k] u,\bar{\partial}^* \lambda(\nabla^c_T)^ku\rangle\\
	=&\sum_{j=0}^{k-1}\int_{\Omega}\lambda Tg((\nabla^c_T)^{j}\bar{\partial}u,(\nabla_{T}^c)^{k-j-1}\bar{\partial}\lambda(\nabla^c_T)^ku)\,dV+\langle\bar{\partial}u, \lambda(\nabla^c_T)^k[\lambda(\nabla^c_T)^k,\bar{\partial}]u\rangle\\&+\langle\bar{\partial}u, [[\lambda(\nabla^c_T)^k,\bar{\partial}],\lambda(\nabla^c_T)^k]u\rangle\\
	&+\sum_{j=0}^{k-1}\int_{\Omega}\lambda Tg((\nabla^c_T)^j\bar{\partial}^* u,(\nabla_T^c)^{k-j-1}\bar{\partial}^* \lambda(\nabla^c_T)^ku)\,dV+\langle\bar{\partial}^* u,\lambda(\nabla^c_T)^k[\lambda(\nabla^c_T)^k,\bar{\partial}^*]u\rangle\\&+\langle\bar{\partial}^* u,[[\lambda(\nabla^c_T)^k,\bar{\partial}^*],\lambda(\nabla^c_T)^k]u\rangle\\
	&+\langle\Box u, \lambda(\nabla^c_T)^k\lambda(\nabla^c_T)^ku\rangle+\langle[\bar{\partial}, \lambda(\nabla^c_T)^k]u,\bar{\partial}\lambda(\nabla^c_T)^ku\rangle+\langle[\bar{\partial}^*,\lambda(\nabla^c_T)^k] u,\bar{\partial}^* \lambda(\nabla^c_T)^ku\rangle\\
	=&\sum_{j=0}^{k-1}\int_{\Omega}\lambda Tg((\nabla^c_T)^{j}\bar{\partial}u,(\nabla_{T}^c)^{k-j-1}\bar{\partial}\lambda(\nabla^c_T)^ku)\,dV+\sum_{j=0}^{k-1}\int_{\Omega}\lambda Tg((\nabla^c_T)^j\bar{\partial}^* u,(\nabla_T^c)^{k-j-1}\bar{\partial}^* \lambda(\nabla^c_T)^ku)\,dV\\
	&-\sum_{j=0}^{k-1}\int_{\Omega}\lambda Tg((\nabla_T^c)^j\Box u,(\nabla_T^c)^{k-j-1}\lambda(\nabla^c_T)^ku)\,dV\\
	&-\sum_{j=0}^{k-1}\int_{\Omega}\lambda Tg((\nabla^c_T)^j\bar{\partial}u, (\nabla^c_T)^{k-j-1}[\lambda(\nabla^c_T)^k,\bar{\partial}]u)\,dV\\&-\sum_{j=0}^{k-1}\int_{\Omega}\lambda Tg((\nabla^c_T)^j\bar{\partial}^* u, (\nabla^c_T)^{k-j-1}[\lambda(\nabla^c_T)^k,\bar{\partial}^*]u)\,dV\\
	&+\langle\lambda(\nabla^c_T)^k\Box u, \lambda(\nabla^c_T)^ku\rangle+2\i\Im\langle[\bar{\partial}, \lambda(\nabla^c_T)^k]u,\bar{\partial}\lambda(\nabla^c_T)^ku\rangle+2\i\Im\langle[\bar{\partial}^*,\lambda( \nabla^c_T)^k] u,\bar{\partial}^* \lambda(\nabla^c_T)^ku\rangle\\
	&+\langle[\bar{\partial},\lambda(\nabla^c_T)^k]u, [\bar{\partial},\lambda(\nabla^c_T)^k]u\rangle+\langle\bar{\partial}u, [[\lambda(\nabla^c_T)^k,\bar{\partial}],\lambda(\nabla^c_T)^k]u\rangle\\
	&+\langle[\bar{\partial}^*,\lambda(\nabla^c_T)^k ]u,[\bar{\partial}^*,\lambda(\nabla^c_T)^k]u\rangle+\langle\bar{\partial}^* u,[[\lambda(\nabla^c_T)^k,\bar{\partial}^*],\lambda(\nabla^c_T)^k]u\rangle    
	\end{align*}
	
	In the last equation, the term\begin{align*}
	&\sum_{j=0}^{k-1}\int_{\Omega}\lambda Tg((\nabla^c_T)^{j}\bar{\partial}u,(\nabla_{T}^c)^{k-j-1}\bar{\partial}\lambda(\nabla^c_T)^ku)\,dV+\sum_{j=0}^{k-1}\int_{\Omega}\lambda Tg((\nabla^c_T)^j\bar{\partial}^* u,(\nabla_T^c)^{k-j-1}\bar{\partial}^* \lambda(\nabla^c_T)^ku)\,dV\\
	&-\sum_{j=0}^{k-1}\int_{\Omega}\lambda Tg((\nabla_T^c)^j\Box u,(\nabla_T^c)^{k-j-1}\lambda(\nabla^c_T)^ku)\,dV-\sum_{j=0}^{k-1}\int_{\Omega}\lambda Tg((\nabla^c_T)^j\bar{\partial}u, (\nabla^c_T)^{k-j-1}[\lambda(\nabla^c_T)^k,\bar{\partial}]u)\,dV\\&-\sum_{j=0}^{k-1}\int_{\Omega}\lambda Tg((\nabla^c_T)^j\bar{\partial}^* u, (\nabla^c_T)^{k-j-1}[\lambda(\nabla^c_T)^k,\bar{\partial}^*]u)\,dV
	\end{align*}
	is bounded by \[\frac{C_{k-1}}{2}\|\bar{\partial}u\|_{W^{k-1}}^2+\frac{C_{k-1}}{2}\|\bar{\partial}^*u\|_{W^{k-1}}^2+\frac{1}{2}\|\bar{\partial}\lambda(\nabla_T^c)^k u\|^2+\frac{1}{2}\|\bar{\partial}^*\lambda(\nabla_T^c)^k u\|^2+\frac{\mu(n-1)}{40000C}\|u\|_k^2.\] This is done by a consecutive integration by parts and the small constant\slash large constant inequality. Clearly $C_{k-1}$ depends on $\mu$ and so is dependent of $\psi$.
	
	We may simply discard the term \[2\i\Im\langle[\bar{\partial}, \lambda(\nabla^c_T)^k]u,\bar{\partial}\lambda(\nabla^c_T)^ku\rangle+2\i\Im\langle[\bar{\partial}^*,\lambda( \nabla^c_T)^k] u,\bar{\partial}^* \lambda(\nabla^c_T)^ku\rangle,\] because it is a purely imaginary. In practice, our expression is necessarily real.
	
	The term \[\langle\bar{\partial}u, [[\lambda(\nabla^c_T)^k,\bar{\partial}],\lambda(\nabla^c_T)^k]u\rangle+\langle\bar{\partial}^* u,[[\lambda(\nabla^c_T)^k,\bar{\partial}^*],\lambda(\nabla^c_T)^k]u\rangle\] is bounded by \[\frac{C_{k-1}}{2}\|\bar{\partial}u\|_{W^{k-1}}^2+\frac{C_{k-1}}{2}\|\bar{\partial}^*u\|_{W^{k-1}}^2+\frac{\mu(n-1)}{40000C}\|u\|^2_k.\]This is done similarly as Straube \cite{St10} (see the discussion after (3.57)).
	
	Thus, we obtain that, assuming $\|u\|_{W^{k-1}}\lesssim\|\Box u\|_{W^{k-1}}$, by Proposition 2.1 of Harrington--Liu \cite{HL19},
	\begin{align*}
	&\|\bar{\partial}\lambda(\nabla^c_T)^ku\|^2+\|\bar{\partial}^*\lambda(\nabla^c_T)^ku\|^2\\
	\leq&C_{k-1}\|\Box u\|^2_{W^{k-1}}+\frac{\mu(n-1)}{20000C}\|u\|^2_k+\frac{1}{2}\|\bar{\partial}\lambda(\nabla_T^c)^k u\|^2+\frac{1}{2}\|\bar{\partial}^*\lambda(\nabla_T^c)^k u\|^2\\&+\|[\bar{\partial},\lambda(\nabla^c_T)^k]u\|^2+\|[\bar{\partial}^*,\lambda(\nabla^c_T)^k]u\|^2,
	\end{align*}
	
	which implies
	\begin{align*}
	&\|\bar{\partial}\lambda(\nabla^c_T)^ku\|^2+\|\bar{\partial}^*\lambda(\nabla^c_T)^ku\|^2\\
	\leq&2C_{k-1}\|\Box u\|^2_{W^{k-1}}+\frac{\mu(n-1)}{10000C}\|u\|^2_k+2\|[\bar{\partial},\lambda(\nabla^c_T)^k]u\|^2+2\|[\bar{\partial}^*,\lambda(\nabla^c_T)^k]u\|^2.
	\end{align*}
	
	Hence, we obtain the following proposition, with $\lambda$, $T$ as introduced before.
	
	\begin{proposition}\label{intem}
		Let $\Omega$ be a bounded pseudoconvex domain with smooth boundary in $\mathbb{C}^n$. Assume $u\in\Dom(\bar{\partial}^*)\cap C^\infty_{(0,q)}(\overline{\Omega})$ and $\mathcal{M}=1$. Let $\mu>0$ be an arbitrary number and suppose $\|u\|_{W^{k-1}}\lesssim\|\Box u\|_{W^{k-1}}$ Then,
		\begin{align*}
		&\|\bar{\partial}\lambda(\nabla^c_T)^ku\|^2+\|\bar{\partial}^*\lambda(\nabla^c_T)^ku\|^2\\
		\leq&2C_{k-1}\|\Box u\|^2_{W^{k-1}}+\frac{\mu(n-1)}{10000C}\|u\|^2_k+2\|[\bar{\partial},\lambda(\nabla^c_T)^k]u\|^2+2\|[\bar{\partial}^*,\lambda(\nabla^c_T)^k]u\|^2.
		\end{align*}
	\end{proposition}

	\section{Commutators}\label{commutatorssection}
	
	In this section, we are going to compute the last two terms in the inequality of the Proposition \ref{intem}.
	From Harrington--Liu \cite{HL19}, we have computed, for $k\in\mathbb{N}$,
	\begin{align*}
	&[\bar{\partial},\lambda(\nabla^c_{T})^k]u\\
	=&\frac{1}{2^q}\sum_{i,I} (\overline{L}_i\lambda)T^kg(u,\omega_{\overline{L}_I})\omega_{\overline{L}_i}\wedge\omega_{\overline{L}_I}+\lambda[\bar{\partial},(\nabla^c_T)^k]u\\
	=&\frac{1}{2^q}\sum_{i,I} (\overline{L}_i\lambda)T^kg(u,\omega_{\overline{L}_I})\omega_{\overline{L}_i}\wedge\omega_{\overline{L}_I}+\frac{k}{2^q}\sum_{i,I} \lambda (\overline{L}_i\psi-2g(L_n, \nabla_{L_i}L_n))T^kg(u,\omega_{\overline{L}_I})\omega_{\overline{L}_i}\wedge\omega_{\overline{L}_I}+\lambda K,
	\end{align*}
	where $\|K\|^2\leq \epsilon\|u\|_{W^k}^2+C_\epsilon(\|\Box u\|_{W^{k-1}}^2+\|u\|_{W^{k-1}}^2)$ for arbitrary $\epsilon>0$.  We set $\epsilon=\frac{\mu(n-1)}{100000C}$ and obtain that $\|K\|^2\leq \frac{\mu(n-1)}{100000C}\|u\|_{W^k}^2+C_\mu(\|\Box u\|_{W^{k-1}}^2+\|u\|_{W^{k-1}}^2)$ for arbitrary $\mu>0$.

	Thus, \begin{align*}
	&\langle[\bar{\partial},\lambda(\nabla^c_{T})^k]u,[\bar{\partial},\lambda(\nabla^c_{T})^k]u\rangle\\
	\leq &\Big(\frac{k^2n}{2^{q-1}}{n\choose q}+1\Big) \sum_{i\notin I}\int_{\Omega} |\overline{L}_i\psi-2g(L_n, \nabla_{L_i}L_n)|^2|g(\lambda(\nabla_{T}^c)^ku,\omega_{\overline{L}_I})|^2\,dV\\
	&+C\left(\|\sum_{i,I} (\overline{L}_i\lambda)T^kg(u,\omega_{\overline{L}_I})\omega_{\overline{L}_i}\wedge\omega_{\overline{L}_I}\|^2+\|\Box u\|_{W^{k-1}}^2+\|u\|_{W^{k-1}}^2\right).
	\end{align*}
	
	Similarly, we have the following.
	\begin{align*}
	&[\bar{\partial}^*,\lambda(\nabla^c_{T})^k]u\\
	=&-\frac{1}{2^{q-1}}\sum_{i,I}(L_i\lambda) T^kg(u, \omega_{\overline{L}_I})\overline{L}_i\lrcorner\omega_{\overline{L}_I}+\lambda[\bar{\partial}^*,(\nabla_T^c)^k]u\\
	=&-\frac{1}{2^{q-1}}\sum_{i,I}(L_i\lambda) T^kg(u, \omega_{\overline{L}_I})\overline{L}_i\lrcorner\omega_{\overline{L}_I}-\frac{k}{2^{q-1}}\sum_{i,I}\lambda (L_i\psi -2g(L_i,\nabla_{L_n}L_n)) T^k g(u, \omega_{\overline{L}_I})\overline{L}_i\lrcorner\omega_{\overline{L}_I}+2\lambda R,
	\end{align*}
	where $\|2R\|^2\leq  \frac{\mu(n-1)}{100000C}\|u\|_{W^k}^2+C_\mu(\|\Box u\|_{W^{k-1}}^2+\|u\|_{W^{k-1}}^2)$ for arbitrary $\mu>0$.
	
	We compute 
	\begin{align*}
	&\langle[\bar{\partial}^*,\lambda(\nabla^c_{T})^n]u,[\bar{\partial}^*,\lambda(\nabla^c_{T})^n]u\rangle\\
	\leq&\Big(\frac{k^2n}{2^{q-1}}{n\choose q}+1\Big)\sum_{i\in I}\int_{\Omega}|L_i\psi -2g(L_i,\nabla_{L_n}L_n)|^2|g(\lambda(\nabla_T^c)^ku, \omega_{\overline{L}_I})|^2\,dV\\
	&+C\left(\|\sum_{i,I}(L_i\lambda) T^kg(u, \omega_{\overline{L}_I})\overline{L}_i\lrcorner\omega_{\overline{L}_I}\|^2+\|\Box u\|_{W^{k-1}}^2+\|u\|_{W^{k-1}}^2\right).
	\end{align*}
	
	Note that, $\lambda\equiv 1$ on $\Sigma$ and $\support(\overline{L}_i\lambda )\cap \partial\Omega\Subset\partial\Omega\backslash\Sigma$.  The term $\sum_{i,I}(L_i\lambda) T^kg(u, \omega_{\overline{L}_I})\overline{L}_i\lrcorner\omega_{\overline{L}_I}$ is supported in a set of strongly pseudoconvex points. By pseudolocal estimate (Theorem 3.6 of Straube \cite{St10}), \[\|\sum_{i,I}(L_i\lambda) T^kg(u, \omega_{\overline{L}_I})\overline{L}_i\lrcorner\omega_{\overline{L}_I}\|^2\lesssim\|\Box u\|^2_{W^{k-1}}.\]
	
	Therefore,  we have that\begin{equation}\label{main}
	\begin{split}
	&\|[\bar{\partial}^*,\lambda(\nabla^c_{T})^k]u\|^2+\|[\bar{\partial},\lambda(\nabla^c_{T})^k]u\|^2\\
	\leq& \Big(\frac{k^2n}{2^{q-1}}{n\choose q}+1\Big)\sum_{i=1}^{n-1}\int_{\Omega}|L_i\psi -2g(L_i,\nabla_{L_n}L_n)|^2|g(\lambda(\nabla_T^c)^ku, \omega_{\overline{L}_I})|^2\,dV+\frac{\mu(n-1)}{10000C}\|\lambda(\nabla_T^c)^k u\|^2\\&+C_\mu\|\Box u\|_{W^{k-1}}^2+C_\mu\|u\|_{W^{k-1}}^2\\
	=&\Big(2k^2n{n\choose q}+2^q\Big)\sum_{i=1}^{n-1}\int_{\Omega}|L_i\psi -2g(L_i,\nabla_{L_n}L_n)|^2g(\lambda(\nabla_T^c)^ku, \lambda(\nabla_T^c)^ku)\,dV+\frac{\mu(n-1)}{10000C}\|\lambda(\nabla_T^c)^k u\|^2\\&+C_\mu\|\Box u\|_{W^{k-1}}^2+C_\mu\|u\|_{W^{k-1}}^2,
	\end{split}
	\end{equation}
	where $C_\mu>0$ is a constant which is dependent of $\mu$.
	
	\section{Proof of Theorem \ref{mainthm}}\label{proof}
	Combining the previous section with Proposition \ref{intem}, we obtain the following proposition. Note that $\|u\|_{W^k}^2\lesssim \|\lambda(\nabla_T^c)^k u\|^2+\|\Box u\|_{W^k}$ if $\|u\|_{W^{k-1}}\lesssim\|\Box u\|_{W^{k-1}}$ is assumed.
	
	\begin{proposition}\label{2interm}
		Let $\Omega$ be a bounded pseudoconvex domain with smooth boundary in $\mathbb{C}^n$. Assume $u\in\Dom(\bar{\partial}^*)\cap C^\infty_{(0,q)}(\overline{\Omega})$ and $\mathcal{M}=1$. Let $\mu>0$ be an arbitrary number and suppose $\|u\|_{W^{k-1}}\lesssim\|\Box u\|_{W^{k-1}}$ Then,
		\begin{align*}
		&\|\bar{\partial}\lambda(\nabla^c_T)^ku\|^2+\|\bar{\partial}^*\lambda(\nabla^c_T)^ku\|^2\\
		\leq&2C_\mu\|\Box u\|_{W^{k-1}}^2+\frac{\mu(n-1)}{1000C}\|\lambda(\nabla_T^c)^k u\|^2\\
		&+2\Big(2k^2n{n\choose q}+2^q\Big)\sum_{i=1}^{n-1}\int_{\Omega}|L_i\psi -2g(L_i,\nabla_{L_n}L_n)|^2g(\lambda(\nabla_T^c)^ku, \lambda(\nabla_T^c)^ku)\,dV,
		\end{align*}
		where $C_\mu>0$ is a constant which is dependent of $\mu$.
	\end{proposition}

	Combining Proposition \ref{2interm} with inequality (\ref{predbar}), we obtain that
	
	\begin{align*}
	&(2\sigma-6-4M^2)\sum_{j=1}^{n-1} \int_{\Omega}|\overline{L}_j\psi-2g(\nabla_{L_n}L_n,L_j)|^2g(\lambda w,\lambda w)\,dV+\frac{\mu}{20}(n-1)\int_{\Omega}g(\lambda w,\lambda w)\,dV\\
	\leq&C\Bigg(2C_\mu\|\Box u\|_{W^{k-1}}^2+\frac{\mu(n-1)}{1000C}\|\lambda(\nabla_T^c)^k u\|^2\\
	&+2\Big(2k^2n{n\choose q}+2^q\Big)\sum_{i=1}^{n-1}\int_{\Omega}|L_i\psi -2g(L_i,\nabla_{L_n}L_n)|^2g(\lambda(\nabla_T^c)^ku, \lambda(\nabla_T^c)^ku)\,dV\Bigg).
	\end{align*}
	
	This implies,
	
	\begin{equation}\label{changesigma}
	\begin{split}
	&\Bigg(2\sigma-6-4M^2-2C\left(2k^2n{n\choose q}+2^q\right)\Bigg)\sum_{j=1}^{n-1} \int_{\Omega}|\overline{L}_j\psi-2g(\nabla_{L_n}L_n,L_j)|^2g(\lambda w,\lambda w)\,dV\\
	&+\frac{\mu}{100}(n-1)\int_{\Omega}g(\lambda w,\lambda w)\,dV\leq2CC_\mu\|\Box u\|_{W^{k-1}}^2.
	\end{split}\end{equation}

	We will use the induction to prove $\|u\|_{W^k}\lesssim\|\Box u\|_{W^k}$ for all non-negative integer $k$.
	
	For $k=0$, it is an easy corollary from the basic estimate of Morrey--Kohn--H\"{o}rmander. See Chen--Shaw \cite{CS01}.
	
	Assume $\|u\|_{W^{k-1}}\lesssim\|\Box u\|_{W^{k-1}}$ is true, we are going to prove $\|u\|_{W^k}\lesssim\|\Box u\|_{W^k}$.
	
	Since \[6+4M^2+2C\left(2k^2n{n\choose q}+2^q\right)\] is independent from the choice of $\psi$, for $k>0$, we can choose $\sigma>0$ big so that  \[2\sigma-6-4M^2-2C\left(2k^2n{n\choose q}+2^q\right)>0.\] We then drop the positive term \[\Bigg(2\sigma-6-4M^2-2C\left(2k^2n{n\choose q}+2^q\right)\Bigg)\sum_{j=1}^{n-1} \int_{\Omega}|\overline{L}_j\psi-2g(\nabla_{L_n}L_n,L_j)|^2g(\lambda w,\lambda w)\,dV.\]Then we obtain that
	\[\int_{\Omega}g(\lambda w,\lambda w)\,dV\leq2\frac{100}{\mu(n-1)}CC_\mu\|\Box u\|_{W^{k-1}}^2,\] which completes the proof.
	
	We want to remark that it is important that $C$ and $M$ are independent of choices of $\psi$. In view of Lemma \ref{properties}, when we change $\sigma>0$ to make  \[2\sigma-6-4M^2-2C\left(2k^2n{n\choose q}+2^q\right)>0,\] the function $\psi$ necessarily varies. If $C$ or $M$ is dependent of $\psi$, it is a risk that it blows up when $\psi$ varies. Consequently, we would never find a $2\sigma$ which is bigger than \[6+4M^2+2C\left(2k^2n{n\choose q}+2^q\right).\] 
	
	The theorem \ref{mainthm} is proved for the case of $\mathcal{M}=1$, i.e., when $\Sigma$ can be covered by one coordinate chart. However, the reader can see Theorem \ref{mainthm} holds for all finite $\mathcal{M}$. Indeed, one can see that $\nabla_T^c$ depends on the choices of coordinate charts but $\nabla_T$ does not. So $(\nabla^c_{T,\alpha})^k-(\nabla_T)^k$ is a $(k-1)$th-order operator on $U_\alpha$. Moreover, one can also see for a partition of unity $\lbrace\chi_\alpha\rbrace$,
	\[
	\sum_{\alpha}[\bar{\partial},\chi_\alpha\lambda(\nabla^c_{T,\alpha})^k]u
	=\sum_{\alpha}\chi_\alpha[\bar{\partial},\lambda(\nabla^c_{T,\alpha})^k]u+\sum_{\alpha}(\bar{\partial}\chi_\alpha)\wedge\lambda((\nabla^c_{T,\alpha})^k-(\nabla_T)^k)u.
	\]
	Thus, the Proposition \ref{2interm} and the inequality (\ref{predbar}) can be handled similarly. This proves an aprori estimate for Theorem \ref{mainthm}.
	
	Moreover, one can use a similar argument of Harrington--Liu \cite{HL19} (see the last section) to prove a genuine estimate. Hence, Theorem \ref{mainthm} is proved.

	\bigskip
	\bigskip
	
	\noindent {\bf Acknowledgments}. The author thanks Dr. Phillip Harrington and Dr. Andrew Raich for fruitful conversations. 
	\printbibliography

\end{document}